\documentclass[12pt, reqno, a4paper]{amsart}
\usepackage{amsmath,mathrsfs,bbm}
\usepackage{amssymb}
\usepackage{color}
\usepackage{lmodern}
\usepackage{dsfont}
\usepackage{todonotes}

\usepackage{hyperref}

\newcommand\NoBlackBoxes{\global\overfullrule0pt}
\NoBlackBoxes

\usepackage{enumitem}
\setitemize{noitemsep,topsep=0pt,parsep=0pt,partopsep=0pt}

\newcommand{\eps}{\varepsilon}

\newcommand{\N}{\mathbb{N}}

\renewcommand{\P}{\mathbb{P}}

\newcommand{\Var}{\mathop{\mathrm{Var}}\nolimits}

\newcommand{\eee}{{\rm e}}

\allowdisplaybreaks
\theoremstyle{plain}
\newtheorem{theorem}{Theorem}[section]

\newtheorem{corollary}[theorem]{Corollary}

\theoremstyle{definition}
\newtheorem{definition}[theorem]{Definition}

\theoremstyle{remark}
\newtheorem{remark}[theorem]{Remark}

\newtheorem{rems}[theorem]{Remarks}

\renewcommand{\P}{{\mathbb{P}}}
\newcommand{\E}{{\mathbb{E}}}
\newcommand{\R}{{\mathbb{R}}}

\renewcommand{\epsilon}{\varepsilon}
\renewcommand{\phi}{\varphi}

\numberwithin{equation}{section}

\begin{document}

\setcounter{page}{1}

\title[Propagation of chaos in the Random field Curie-Weiss model]{Propagation of chaos in the Random field Curie-Weiss model}
%
%
%
\author[Zakhar Kabluchko]{Zakhar Kabluchko}
\address[Zakhar Kabluchko]{Fachbereich Mathematik und Informatik,
Universit\"at M\"unster,
Einsteinstra\ss e 62,
48149 M\"unster,
Germany}

\email[Zakhar Kabluchko]{zakhar.kabluchko@uni-muenster.de}

\thanks{Research of both authors was
funded by the Deutsche Forschungsgemeinschaft (DFG, German Research Foundation) under Germany 's Excellence Strategy
EXC 2044-390685587, Mathematics M\"unster: Dynamics-Geometry-Structure}

\author[Matthias L\"owe]{Matthias L\"owe}
\address[Matthias L\"owe]{Fachbereich Mathematik und Informatik,
Universit\"at M\"unster,
Einsteinstra\ss e 62,
48149 M\"unster,
Germany}

\email[Matthias L\"owe]{maloewe@math.uni-muenster.de}


\date{\today}



\newcommand{\wlim}{\mathop{\hbox{\rm w-lim}}}
\newcommand{\na}{{\mathbb N}}
\newcommand{\re}{{\mathbb R}}

\newcommand{\vep}{\varepsilon}
\newcommand{\be}{\begin{equation}}
\newcommand{\ee}{\end{equation}}

\begin{abstract}
We prove propagation of chaos in the Random field mean-field Ising model, also known ad the Random field Curie-Weiss model.
We show that in the paramagnetic phase, i.e.\ in the regime where temperature and distribution of the external field admit a
unique minimizer of the expected Helmholtz free energy, propagation of chaos holds. By the latter we mean that the finite-dimensional
marginals of the Gibbs measure converge towards a product measure with the correct expectation as the system size goes to infinity.
This holds independently of whether the system is in a high-temperature phase or at a phase transition point. If the
Helmholtz free energy possesses several minima, there are several possible equilibrium measures. In this case, we show that the
system picks one of them at random (depending on the realization of the random external field) and propagation of chaos with respect to
a product measure with the same marginals as the one randomly picked
holds true. We illustrate our findings in a simple example.
\end{abstract}

\maketitle

\section{Introduction and main results}
This note is devoted to propagation of chaos in one of the easiest disordered mean-field models, the Random field Curie-Weiss model.
Propagation of chaos was introduced by Kac \cite{Kac_foundations} in attempt to justify Boltzmann's ''Stoßzahlenansatz'' that had led to
fierce debates in the physics community since it seems to contradict microscopic reversibility. Kac introduced the idea that for a
system
that evolves in time chaos should be propagated, i.e.\ if at time $0$ the finite marginal distributions of a system are product measures in the
thermodynamic limit, this should still be true for the time-evolved system. For mean-field Gibbs measures this was shown to
follow from the fact the extremal Gibbs measures locally look like product measures, i.e.\
that any finite subset of spins forms a family of independent random variables in the thermodynamic limit. This is even true, if
one considers marginals that grow sufficiently slow with the number of spins \cite[Theorem 3]{BAZ_chaos}.

We will analyze exactly this property for the Random field Curie-Weiss model. This is not itself a mean-field model in the sense of
\cite{BAZ_chaos}, since its energy function cannot be written as a functional of the empirical mean of the spin-values, nor of the
empirical measure. However, it is one of the easiest disordered variants of one of the easiest models of ferromagnetism. To show
propagation of chaos for this model, we will introduce a new technique, which also might be useful for deterministic, i.e.\
non-disordered models.

To be more precise consider the following sequence of random functions on the set $\{-1, +1\}^N, N \in \N$:
$$
H_N(\sigma):= -\frac 1 {2N} \sum_{i,j=1}^N \sigma_i \sigma_j -\sum_{i=1}^N h_i \sigma_i, \qquad \sigma=(\sigma_i)_{i=1}^N \in \{-1, +1\}^N.
$$
Here $(h_i)$ are i.i.d. random variables, of which we assume that $\E_h h_i^2=\E h_i^2$ exists.
(We will always denote the probability measure governing the randomness of the $(h_i)$ by $\P_h$ and the corresponding
expectation by $\mathbb E_h$; if there is no danger of confusion we will drop the index $h$.)

With $(H_N)$ we associate a sequence of Gibbs measures
$$
\mu_N(\sigma):= \frac{\exp\left(-\beta H_N(\sigma)\right)}{Z_N}, \qquad \sigma\in \{-1, +1\}^N,
$$
where, as usual, $\beta>0$ is the inverse temperature and
$$Z_N=\sum_{\sigma'}\exp\left(-\beta H_N(\sigma')\right)$$
is the partition function. Note that $\mu_N$ and $Z_N$ are random probability measures and random constants, respectively, through
their dependence on $(h_i)$. A special case of this situation is the classical Curie-Weiss model, where $h_i \equiv h$ is a deterministic
number (see e.g.\ \cite{EllisEntropyLargeDeviationsAndStatisticalMechanics,bovierbook,Velenik_book} for a survey of the vast number
of results obtained in this model).

These models have been studied intensively over the
last decades, see e.g.\ \cite{Aha78,SW85,MPF91,AGZ92,MP98,FMP00,BBI09,LM12,LMT13, CK18}, the references
therein and many others. Many of the above articles assume the
local external fields to be symmetrically Bernoulli distributed or to be bounded.
We will not need to restrict ourselves to these assumptions, even though, as we will see in the example in Section 4,
already symmetrically distributed Bernoulli variables $(h_i)$ provide a much richer phase diagram than the usual Curie-Weiss model.

As mentioned above, in this note we will focus on the so-called propagation of chaos. That means that for $k\in \N$ not depending on $N$
we want to approximate the $k$-marginal of $\mu_N$, which we call $\mu_{N,k}$ by a product measure. Note that, other than in the
classical
Curie-Weiss model, the so-called spin variables $(\sigma_i)$ are not exchangeable, hence it does matter which $k$-marginal
$(\sigma_{i_1}, \ldots, \sigma_{i_k})$ we consider.
On the other hand, the dependence on the sites $i_1, \ldots, i_k$ is just via the realizations $h_{i_1}, \ldots, h_{i_k}$, so we may
as well just take $(i_1, \ldots, i_k)=(1, \ldots,k)$ to get the full picture. $\mu_{N,k}$ will therefore henceforth denote the
marginal distribution of the first $k$ spins.

As announced we want to compare $\mu_{N,k}$ to a product measure which we
will call $\varrho^{(k)}$. In the classical Curie-Weiss model $\varrho^{(k)}$ can also, more
generally, be a mixture of product measures. This situation occurs in the low temperature regime of this model (with zero external magnetic field).
However, as we will see, in the random field Curie-Weiss model with $\E h_i^2 \neq 0$ this situation
does not occur with probability approaching $1$ (Theorem~\ref{theo:main2}).
Of course, the above statement provokes
the question in which metric this comparison will be done. In alignment with \cite{BAZ_chaos} we will choose the relative entropy or
Kullback-Leibler ''distance'' defined as
\be
H(\mu_{N,k}\vert \varrho^{(k)})= \begin{cases} \int \log \left(\frac{d\mu_{N,k}}{d\varrho^{(k)}}\right)d\mu_{N,k}
& \mbox{if } \mu_{N,k} \ll \varrho^{(k)}\\
\infty & \mbox{otherwise.}
\end{cases}
\ee
(even though we will even have stronger statements).
Of course, in our case the above integrals are actually sums since the measures $\mu_{N,k}$ and $\varrho^{(k)}$ live on $\{+1,-1\}^k$.
Note that relative entropy lacks the triangle inequality and therefore is
not a real distance. However, due to Pinsker's inequality
$$
d_{TV}(\mu,\nu) \le \sqrt{\frac 12  H(\mu \vert \nu)}
$$
for any two probability measures $\mu, \nu$ on a measurable space $(\Omega, \mathcal{A})$ convergence in relative entropy implies
convergence in total variation $d_{TV}$.

An important role in the formulation (and proof) of our results will be played by the following function, which can be thought of as a
relative of the Helmholtz free energy:
\be \label{eq:HFE}
G(y):= -\frac{y^2}2+ \E_h \left[\log \cosh\left(\sqrt \beta y+\beta h\right)\right].
\ee
Finally let us agree on the following definition
\begin{definition}
Let $f:\R \to\R$ be a sufficiently often differentiable function and let $x_0$ be a global maximum point of $f$. As usually, we say that
$x_0$ is non-degenerate, if $f'(x_0)=0$ and $f''(x_0)<0$.

For $n \in \N, n\ge 2$ we say that $x_0$ is $n$-degenerate, if $f'(x_0)=f''(x_0)=\ldots = f^{(2n-1)}(x_0)=0$ and $f^{(2n)}(x_0)<0$.
\end{definition}

With these definitions our main results read as follows:
\begin{theorem}\label{theo:main1}
Assume that the function $G(y)$ has a unique, non-degenerate global maximum in some point $y_0$.
Assume that the $(h_i)_{i \in \N}$ are i.i.d.\ and have a finite second moment. 

For any fixed realization of $(h_i)_{i \in \N}$ and any fixed $k\in \N$ consider a probability measure on $\{+1,-1\}^k$ defined by
\begin{align}
\varrho^{(k)}(\sigma_1, \ldots, \sigma_k)
&:=\prod_{i=1}^k\frac{\exp\left(\sigma_i\left(\sqrt \beta y_0+\beta h_i \right)\right)}
{2\cosh\left(\sqrt \beta y_0+\beta h_i \right)}\label{eq:defrho}\\
&=
\prod_{i=1}^k\frac{\exp\left(\sigma_i\left(\sqrt \beta y_0+\beta h_i \right)\right)}{\exp\left(\sqrt \beta y_0+\beta h_i \right)+\exp\left(-\sqrt \beta y_0-\beta h_i\right)}.\notag
\end{align}
Then, for a.e.\ realization of $(h_i)_{i \in \N}$, we have $\mu_{N,k}(\sigma_1, \ldots, \sigma_k)\to  \varrho^{(k)}(\sigma_1, \ldots, \sigma_k)$ for all $(\sigma_1,\ldots,\sigma_k) \in \{-1,+1\}^k$ and also
\be
H\left(\mu_{N,k} \vert \varrho^{(k)}\right) \to 0 \qquad \text{ as } N\to\infty.
\ee
\end{theorem}

For degenerate maxima we have the same result, yet with a different proof:
\begin{theorem}\label{theo:main1b}
Assume that $G$ has a unique, $n$-degenerate global maximum in some point $y_0$ for some $n\ge 2$.
Assume that the $(h_i)_{i \in \N}$ are i.i.d.\ random variables with finite second moment.
For any fixed realization of $(h_i)_{i \in \N}$ and any fixed $k\in \N$ let $\varrho^{(k)}$ be defined as in \eqref{eq:defrho}.
Then, the same conclusions as in Theorem~\ref{theo:main1} hold.
\end{theorem}

For the case of several (isolated) maxima of the function $G$ we prove the following
\begin{theorem}\label{theo:main2}
Assume that $G$ has several non-degenerate isolated global maxima in some points $y_0, \ldots, y_s$ where $s\ge 1$ and that $(h_i)_{i \in \N}$ are i.i.d.\ random variables with finite second moment satisfying $\E h_i^2 \neq 0$.
For any fixed realization of $(h_i)_{i \in \N}$, any fixed index $j=0, \ldots, s$, and any fixed $k$ let
$$
\varrho^{(k)}_j(\sigma_1, \ldots, \sigma_k):=
\prod_{i=1}^k\frac{\exp\left(\sigma_i\left(\sqrt \beta y_j+\beta h_i \right)\right)}
{2\cosh\left(\sqrt \beta y_j+\beta h_i \right)}.
$$
Then, for almost every realization of the $(h_i)_{i \in \N}$ and every
$N$
there is an
index $J\in \{0, \ldots, s\}=J_N^{(k)}(h_1,\ldots, h_N)$ (depending on that realization) such that
for all $\sigma_1, \ldots, \sigma_k \in \{-1,+1\}$
$$
\left|\mu_{N,k}(\sigma_1, \ldots, \sigma_k)-\varrho^{(k)}_{J_N^{(k)}}(\sigma_1, \ldots, \sigma_k)\right| \to 0
$$
and, consequently,
\be
H\left(\mu_{N,k} \vert \varrho^{(k)}_{J_N^{(k)}}\right) \to 0 \qquad \text{ as } N \to \infty.
\ee
Both limits above should be understood in the sense of convergence in probability.
\end{theorem}
\begin{rems}
\begin{enumerate}
\item Note that the situation described in Theorems \ref{theo:main1} and \ref{theo:main2} differs from the {\it{increasing}} propagation of
chaos when $k$ may depend on $N$ proved for mean-field Gibbs measures in \cite{BAZ_chaos}. In the latter, in the general situation it is
important that the maxima of $G$ or the free energy, respectively, are non-degenerate. The authors in \cite{BAZ_chaos} also treat a
specific example (the Curie-Weiss model at $\beta=1$) of a unique degenerate maximum. There, the fact that the maximum is degenerate is
reflected in a slower increase of $k$ with $N$, i.e.\ $k=o(\sqrt N)$ rather than $k=o(N)$. In our situation $k$ is independent of $N$
and this latter phenomenon cannot occur. However, the degeneracy of the maximum is reflected in the proof. Moreover, concerning increasing propagation of chaos notice the part \eqref{rem:incr} of this remark below.

\item One might also want prove a version of Theorem \ref{theo:main2} with several degenerate maxima. However, this generalization is
not completely straightforward. The reason is the remark after equation \eqref{eq:innerinttheo3}.
Since in our main example such a situation
does not occur we refrain from trying to prove such a version.
\item  \label{rem:incr} A slight modification of the proof of Theorem~\ref{theo:main1} (see Remark~\ref{rem:increasing}) shows that in this situation we also have increasing propagation of chaos in the sense that $k= k(N)$ might increase with $N$  but not faster than $o(N^{\frac 12 - \vep})$ (for any fixed $\vep >0$).  The result takes the form
    $$
    \lim_{N\to\infty} \frac{\mu_{N,k}(\sigma_1, \ldots, \sigma_k)}{\varrho^{(k)}(\sigma_1, \ldots, \sigma_k)} = 1
    $$
    uniformly in $(\sigma_1,\ldots,\sigma_k) \in \{-1,+1\}^k$, and consequently also $H(\mu_{N,k} \vert \varrho^{(k)}) \to 0$  as $N\to\infty$.
\item  Note that Theorem \ref{theo:main2} is of a different nature than the situation where $h_i=0$ a.s., i.e.\ the case of the classic
Curie-Weiss model with zero external magnetic field. In the latter~\cite{BAZ_chaos}, in the low-temperature region $\beta>1$ there is also (increasing) propagation of chaos, but
the limiting measure is a mixture of two product measures rather than a  pure product measure (chosen at random from a finite collection of such measures).

\item Notice that in
\cite{JKLM23} also the question when the chaos stops to propagate in the classic Curie-Weiss model was discussed. More precisely, for blocks of a size $k=k(N)$ such that
$k(n)/N \to \alpha >0$ the non-zero limit between the total variation distance of the marginals of the Gibbs measure and the appropriate product measure was identified.
 \end{enumerate}
\end{rems}

\begin{rems}
\begin{enumerate}
\item
The limiting measure of the measures $\mu_{N,k}$ as $N \to \infty$, i.e.\ the measures $\varrho_j^{(k)}$ in Theorem \ref{theo:main2}
are related to the empirical meta-states in the Random Field Curie-Weiss model as constructed in \cite{Ku97,Ku98}.
As a matter of fact,  in the Random Field Curie-Weiss model with dichotomous random field and in the region where $s=1$, the emprical
metastates are shown to be random mixtures of infinite product measures. The finite dimensional marginals of these product measures
in our language are
$$
\prod_{i=1}^k\frac{\exp\left(\sigma_i\left(\sqrt \beta y_j+\beta h_i \right)\right)}
{2\cosh\left(\sqrt \beta y_j+\beta h_i \right)}, \quad j=0,1.
$$
\item
Another disordered model for which propagation of chaos was shown to be true is the Hopfield model with a sufficiently small number of patterns
and either at high temperatures or conditioned on the overlap being close to one of the patterns. This result follows from a deep
analysis of the structure of the free energy close to the patterns \cite[Theorem 81.5]{BG98book}.
\end{enumerate}
\end{rems}
We organize the rest of this note in the following way. In Section 2, for the reader's convenience, we collect some results related
to the Marcinkiewicz-Zygmund strong law of large numbers and the Central Limit Theorem in certain Banach spaces
that we will use in the proof of the main theorems. In Section 3 we will prove Theorems \ref{theo:main1}, \ref{theo:main1b},
and \ref{theo:main2}.
We will give the proof of Theorem \ref{theo:main1} in detail. The proofs of Theorems \ref{theo:main1b} and \ref{theo:main2} follow
in large
parts the same idea as the proof
of Theorem \ref{theo:main1} and we just indicate how to treat the degeneracy (for Theorem \ref{theo:main1b}) or how the random choice
of the measures
$\varrho_j^{(k)}$ arises, rather than a mixture of these measures (for Theorem \ref{theo:main2}). Section~\ref{sec:example} will be devoted to the
example of
a Random Field Curie-Weiss model with dichotomous and centered random field.

\section{Auxiliary results}
In Section 3 we will need the following result related to the  Marcinkiewicz-Zygmund strong law of large numbers for random functions.
\begin{theorem}\cite[Theorem~3.1]{deAcosta81}\label{theo:deacosta}
Let $B$ be a separable Banach space, and let $1 < p < 2$. Let $(X_j)_{j=1}^\infty$ be independent identically distributed
$B$-valued random variables with $\E \lVert X_1 \rVert^p < \infty$. Define $S_n = \sum_{j=1}^n X_j$. Then
\be
\frac{S_n}{n^{1/p}} \to 0 \qquad \mbox{in probability, if and only if} \quad \frac{S_n}{n^{1/p}} \to 0 \mbox{ almost surely.}
\ee
\end{theorem}

The second result we need  is a Central Limit Theorem for random continuous functions. In order to formulate it, let $(S,d)$ be a compact metric
space, let $\mathcal{C}(S)$ be the Banach space of continuous functions on $S$ (with the supremum norm), and let $\mathcal{C}(S)^*$ be the space of continuous, linear functionals
on it. Moreover, for $\vep >0$ let $N(S,\vep)$ denote the minimum number of balls of radius at most $\vep$ that cover $S$ and set
$$
H(S,\vep):=\log N(S, \vep).
$$
Then, for a sequence of i.i.d. random variables $(X_n)_{n=1}^\infty$ with values in $\mathcal{C}(S)$ such that
\be \label{eq:centered}
\E f(X_1)=0 \qquad \mbox{for all } f\in \mathcal{C}(S)^*
\ee
and
\be \label{eq:finvar}
\sup_{t \in S} \E \left[(X_1(t))^2\right] <\infty
\ee
the following theorem holds
\begin{theorem}\cite[Theorem 1 and Corollary 1]{JainMarcus} \label{theo:JD}
In the above situation assume that there is a non-negative random variable $M$ with $\E[M^2]< \infty$ such that
\be
\vert X(s,\omega)-X(t,\omega)\vert \le M(\omega) d(s,t)
\ee
and
\be
\sum_{n=1}^\infty 2^{-n} H(S,2^{-n})^{1/2} < \infty.
\ee
Then the distributions of $\frac 1 {\sqrt n}\left(X_1 +\ldots + X_n\right)$ converge weakly to a Gaussian measure on $\mathcal{C}(S)$.
\end{theorem}

Let us now turn to the proof of our main theorems.

\section{Proof of the main Theorems}
We start with the simplest situation of a unique, non-degenerate global maximum. This already contains many of the most important ideas.
\begin{proof}[Proof of Theorem \ref{theo:main1}]
Let us start by introducing some notation that is also used in the proofs of the
other two main theorems. For $\sigma \in\{-1,+1\}^N$ let us denote by $\tilde \sigma=\tilde \sigma_k$ the projection of
$\sigma$ onto its first $k$ coordinates, thus $\tilde \sigma:=\tilde \sigma(\sigma):=(\sigma_1, \ldots, \sigma_k)$.
Notice that with this notation
\be \label{eq:mutilde}
\mu_{N,k}(\tilde \sigma)= \sum_{\sigma_{k+1}, \ldots \sigma_N\in \{-1,+1\}}\frac{\exp\left(-\beta H_N(\sigma)\right)}{Z_N}.
\ee
We will rewrite both, the numerator and the denominator of the right-hand side, by a Hubbard-Stratonovich transformation. Using this well-known
technique we obtain
\begin{align*}
Z_N &= 2^{N} 2^{-N} \sum_{\sigma'\in\{-1,+1\}^N} \exp\left(\frac{\beta}{2N}\left(\sum_{i=1}^N \sigma'_i \right)^2+\sum_{i=1}^N \beta
\sigma'_i h_i\right)\\
&=2^N \frac 1 {\sqrt{2\pi}}\int_{-\infty}^\infty 2^{-N} \sum_{\sigma'\in\{-1,+1\}^N}  \exp\left(-\frac{x^2}2+\sqrt{\frac{\beta}N}
\sum_{i=1}^N \sigma'_i x+\sum_{i=1}^N \beta \sigma'_i h_i\right)dx\\
&=2^N \frac 1 {\sqrt{2\pi}}\int_{-\infty}^\infty \exp\left(-\frac{x^2}2+\sum_{i=1}^N \log\cosh\left(\sqrt{\frac{\beta}N}
x+\beta h_i\right)\right)dx\\
&=2^N \sqrt{\frac N {2\pi}}\int_{-\infty}^\infty \exp\left(-N\frac{y^2}2+N\frac 1 N\sum_{i=1}^N \log\cosh\left(\sqrt{\beta}
y+\beta h_i\right)\right)dy
\end{align*}
by substituting $\frac{x}{\sqrt N}=y$ in the last step. Hence, if we define
$$
G_N(y):= -\frac{y^2}2+\frac 1 N\sum_{i=1}^N \log\cosh\left(\sqrt{\beta}
y+\beta h_i\right)
$$
we have seen that
$$
Z_N= 2^N \sqrt{\frac N {2\pi}}\int_{-\infty}^\infty \exp\left(N G_N(y)\right) dy.
$$
In the same way we can treat the nominator in \eqref{eq:mutilde} to obtain
\begin{align*}
& \sum_{\sigma_{k+1}, \ldots \sigma_N\in \{-1,+1\}}\exp\left(-\beta H_N(\sigma)\right)\\
=&2^{N-k} \frac 1 {\sqrt{2\pi}}\int_{-\infty}^\infty 2^{-(N-k)} \sum_{\sigma_{k+1}, \ldots \sigma_N\in \{-1,+1\}}  \exp\left(-\frac{x^2}2+\sqrt{\frac{\beta}N}
\sum_{i=1}^N \sigma_i x+\sum_{i=1}^N \beta \sigma_i h_i\right)dx\\
=&2^{N-k} \sqrt{\frac 1 {2\pi}}\int_{-\infty}^\infty \exp\left(-\frac{x^2}2+\sum_{i=k+1}^N \log\cosh\left(\sqrt{\frac{\beta}N}
x+\beta h_i\right)\right)
\exp\left(\sum_{i=1}^k\left(\sqrt{\frac{\beta}N} x \sigma_i + \beta h_i \sigma_i\right)\right) dx\\
= &2^{N-k} \sqrt{\frac N {2\pi}}\int_{-\infty}^\infty \exp\left(-N\frac{y^2}2+\sum_{i=k+1}^N \log\cosh\left(\sqrt{\beta}
y+\beta h_i\right)\right)
\exp\left(\sum_{i=1}^k(\sqrt \beta y \sigma_i+\beta h_i \sigma_i)\right) dy\\
= &2^{N} \sqrt{\frac N {2\pi}}\int_{-\infty}^\infty \exp\left(N G_N(y)\right)g(y) dy
\end{align*}
Here we have introduced the function
$$
g(y):= \prod_{i=1}^k \frac{\exp\left(\sigma_i\left(\sqrt \beta y+\beta h_i \right)\right)}
{2\cosh\left(\sqrt \beta y+\beta h_i \right)}.
$$
Note that $g$ is indeed solely a function of $y$ because $\sigma_1, \ldots, \sigma_k$ and $h_1, \ldots, h_k$ are kept fixed for the purpose of this proof.
Moreover, as a function in $y$ it is arbitrarily often differentiable and bounded.

Therefore
\be \label{eq:rep_marg}
\sum_{\sigma_{k+1}, \ldots \sigma_N \in \{-1,+1\}}\frac{\exp\left(-\beta H_N(\sigma)\right)}{Z_N}=
\frac{ \int_{-\infty}^\infty \exp\left(N G_N(y)\right) g(y) dy}{\int_{-\infty}^\infty \exp\left(N G_N(y)\right) dy}.
\ee
It is natural to expect that we can treat these integrals with Laplace's method. This is, however, not straightforward, because
first we need to restrict the domain of integration to a finite interval and second, more importantly, the function $G_N$ in the exponent
depends on $N$.

To solve the first of these problems note that for each fixed $y\in \mathbb R$ the functions $G_N(y)$ converge almost surely to $G(y)$ as defined in \eqref{eq:HFE}, by the law of large numbers.  Also observe
that for this theorem we assume that $G(\cdot)$ has a unique non-degenerate maximum in $y_0$. Finally, we remark
that $G(y)$ diverges to $-\infty$, when $|y| \to \infty$. But then
also $G_N(y)$ diverges almost surely to $-\infty$, when $|y| \to \infty$. Indeed, the reason that $G(y) \to -\infty$ as
$|y| \to \infty$ is, of course, that for $|y|$ large enough
\begin{align}
\left| \E_h \log \cosh(\sqrt \beta y+ \beta h ) \right| & \le \log 2 +\log \cosh(\sqrt \beta y)+\left|\E_h\log \cosh(\beta h)\right|\\
& \le \log 2+ |\sqrt \beta y|+ |\E_h \beta h | \le  y^2/4
\end{align}
by the addition theorem for $\cosh$, the fact that $|\sinh(x)| \le \cosh(x)$, and $\cosh(x)\le e^{|x|}$. In the same way,
we estimate
$$
\left|\frac 1N \sum_{i=1}^N \log \cosh(\sqrt \beta y+ \beta h_i y)\right|\le
\log 2+ |\sqrt \beta y|+ \left| \frac 1N \sum_{i=1}^N h_i \beta \right|.
$$
It follows that, for sufficiently large $|y|$ and all $N \ge N_0(\omega)$,
$$
G_N(y) \leq -y^2/4 + C(\omega).
$$

Hence, there is an interval $[a,b]$ (without loss of generality containing $y_0$, with  $a=y_0-r<0$ and $b=y_0+r>0$ for some
sufficiently large $r>0$) such that for any function $0\leq g(y)\leq 1$ 
\be \label{eq:intrhs}
\int_{-\infty}^\infty \exp\left(N G_N(y)\right) g(y)dy = \int_{a}^{b} \exp\left(N G_N(y)\right) g(y)dy  + O\left(\int_{\R \backslash [a,b]} e^{-N (y^2/4 -C(\omega))}dy\right).
\ee
For the first integral on the right-hand side of \eqref{eq:intrhs} we now adopt a version of Laplace's method.

To this end we fix some $\vep \in (0, 1/2)$ and we decompose
$$[y_0-r,y_0+r]= A_N \mathop{\dot{\cup}} \left(y_0-N^{\vep-1/2}, y_0+N^{\vep-1/2}\right),$$
where, of course,
\be
\label{eq:AN}
A_N:=[y_0-r,y_0+r]\setminus \left(y_0-N^{\vep-1/2}, y_0+N^{\vep-1/2}\right).
\ee
Let us first analyze the contribution of $A_N$ to the first integral on the right-hand side of \eqref{eq:intrhs}.
To this end write
\be
\label{eq:intGGN}
\int_{A_N} \exp\left(N G_N(y)\right) g(y)dy= \int_{A_N} \exp\left(N G(y)+N \Delta_N(y)\right) g(y)dy
\ee
where
\begin{align}
\Delta_N(y)
&:=
G_N(y) -G(y) \notag \\
&=
\frac 1N \sum_{i=1}^N \log \cosh\left(\sqrt \beta y+\beta h_i \right)
-\E_h \left[\log \cosh\left(\sqrt \beta y+\beta h\right)\right]=:\frac 1N \sum_{i=1}^N X_i(y).  \label{eq:Delta}
\end{align}
Notice that for each $y\in \R$ we have $\E_h X_i(y)=0$ and $\E_h(X_i^2(y)) <\infty$ since $h_i$ has finite second moment.

By a simple Taylor expansion in $y_0$:
\begin{equation}\label{eq:taylor_Delta_N}
\Delta_N(y)=\Delta_N(y_0)+\Delta'_N(y_0)(y-y_0)+\frac 12 \Delta''_N(\xi)(y-y_0)^2
\end{equation}
for some $\xi = \xi(y) \in [y_0,y]$ or $\xi \in [y,y_0]$. 

Notice that
\begin{align*}
\Delta'_N(y_0)=& \frac 1N \sum_{i=1}^N X'_i(y_0)\\
= &
\frac 1N \sum_{i=1}^N \sqrt \beta\tanh\left(\sqrt \beta y_0+\beta h_i\right)-\E_h\left[\sqrt
\beta\tanh\left(\sqrt \beta y_0+\beta h \right)\right].
\end{align*}
Since $|\tanh z|\leq 1$,  $(X_i'(y_0))_i$ have expectation $0$ and a finite variance. Thus, by the Law of
the Iterated Logarithm,
\begin{equation}\label{eq:LIL1}
|\Delta'_N(y_0)| \le \frac{\log N}{\sqrt N}
\ee
for all $N \ge N_1= N_1(\omega)$ with a $N_1$ that depends on the realization of the $(h_i)$.

Similarly,
\begin{align*}
\Delta''_N(y)=& \frac 1N \sum_{i=1}^N X''_i(y)\\
= &
\frac 1N \sum_{i=1}^N \beta \mathrm{sech}^2\left(\sqrt \beta y+\beta h_i \right)-\E_h\left[\beta
\mathrm{sech}^2\left(\sqrt \beta y+\beta h\right)\right]
\end{align*}
and note that also the $X''_i(y)$ satisfy $\E_h(X''_i(y))=0$ by construction and $\E_h\left[\left(X''_i(y)\right)^2\right]< \infty$.
This again implies by
Law of the Iterated Logarithm that for any fixed $y \in [y_0-r,y_0+r]$ and any $\vep>0$
$$
\vert \Delta''_N(y) \vert \le \frac{N^{\vep}}{\sqrt N}
$$
for $N \ge N_2(\omega,y)$, $\P_h$-almost surely.

However, this time we need such a result uniformly in $y \in (y_0-r,y_0+r)$ since $\xi$ may depend on $N$ and we do not know where it is
located. To this end,
note that the functions $X_i''(\cdot)$ on the compact metric space $S=[a,b]$ satisfy the assumptions of Theorem \ref{theo:JD}.
Indeed, they are i.i.d. and, by construction, they are centered, hence satisfy \eqref{eq:centered}. Moreover, for each $y\in[a,b]$ we
have seen that $\E_h\left[\left(X''_i(y)\right)^2\right]< \infty$ and the variances are continuous in $y$, hence also
\eqref{eq:finvar} is satisfied.  Also notice that
$$
y \mapsto \beta \mathrm{sech}^2\left(\sqrt \beta y+\beta h_i\right)
$$
is Lipschitz-continuous for every realization of the $h_i$. Indeed, since $\frac d{dx}(\mathrm{sech(x)}^2)\le 1$ the Lipschitz-constant
of $y \mapsto \beta \mathrm{sech}^2\left(\sqrt \beta y+\beta h_i\right)$
is at most $\beta$, so that $M$ from Theorem~\ref{theo:JD} is square-integrable.

Finally, using the notation of Section 2, of course, $N([y_0-r,y_0+r], 2^{-n})\le 2r\cdot 2^n + 2$ such that
$$
\sum_{n=1}^\infty 2^{-n}H([y_0-r,y_0+r], 2^{-n})^{\frac 12}\le \mathrm{Const} \sum_{n=1}^\infty n 2^{-n}<\infty.
$$
Therefore, from Theorem \ref{theo:JD} we see that the random functions $y\mapsto \frac{1}{\sqrt N} \sum_{i=1}^N X''_i(y)
=\sqrt N \Delta_N''(y)$ converge weakly  to a centered Gaussian process on $\mathcal{C}([y_0-r,y_0+r])$. This implies for any $\vep'>0$ that
$$
\frac{1}{N^{1/2+\vep'}} \sum_{i=1}^N X''_i(\cdot)=N^{1/2-\vep'} \Delta_N''(\cdot) \to 0 \quad \mbox{in probability}.
$$
Using Theorem \ref{theo:deacosta} for the Banach space of continuous functions $\mathcal{C}([y_0-r,y_0+r])$ this implies that also
$$
\frac{1}{N^{1/2+\vep'}} \sum_{i=1}^N X''_i(\cdot)=N^{1/2-\vep'} \Delta_N''(\cdot) \to 0 \quad \mbox{$\P_h$-almost surely.}
$$
Hence for almost all realizations of $(h_i)=(h_i(\omega))$ and all $N \ge N_2(\omega)$ we obtain
\begin{equation}\label{eq:sup_Delta_N''}
\sup_{\xi \in [y_0-r,y_0+r]}|\Delta_N''(\xi)| \le \frac{N^{\vep'}}{\sqrt N}.
\end{equation}
Getting back to \eqref{eq:intGGN} and using that $G$ has a non-degenerate maximum in $y_0$ (and that $0\leq g(y)\leq 1$) we thus obtain for some $\chi>0$
\begin{align}\label{eq:longcomp}
&\int_{A_N} \exp\left(N G_N(y)\right) g(y)dy \\
=& \int_{A_N} \exp\left(N G(y)
+N\left(\Delta_N(y_0)+\Delta'_N(y_0)(y-y_0)+\frac 12 \Delta''_N(\xi)(y-y_0)^2\right)\right)g(y)dy\notag\\
\leq & \int_{A_N} \exp\left(N \left(G(y_0)-\frac{\chi}2 (y-y_0)^2\right)\right. \notag\\
& \qquad\qquad\left.
+N\left(G_N(y_0)-G(y_0)\right)+N\Delta'_N(y_0)(y-y_0)+\frac N2 \Delta''_N(\xi)(y-y_0)^2\right)dy\notag\\
=&e^{N G_N(y_0)} \int_{A_N} \exp\left(-N\frac{\chi}2 (y-y_0)^2
+N\Delta'_N(y_0)(y-y_0)+\frac N2 \Delta''_N(\xi)(y-y_0)^2\right)dy\notag\\
\le &\frac 1{\sqrt N} e^{N G_N(y_0)}  \left(\int_{-\infty}^{-N^\vep} \exp\left(-\frac{\chi}2 x^2
+\sqrt N \Delta'_N(y_0)x+\frac 12 \Delta''_N(\xi)x^2\right)dx\right.
\notag\\& \hskip 5.5cm\left.+\int_{N^\vep}^{\infty} \exp\left(-\frac{\chi}2
 x^2
+\sqrt N \Delta'_N(y_0)x+\frac 12 \Delta''_N(\xi)x^2\right)dx\right)\notag
\end{align}
by substituting $x=\sqrt N (y-y_0)$.

Since we have established that $|\Delta''_N(\xi)| \leq  N^{\vep'}/\sqrt N$ uniformly in $\xi \in A_N$ and that $|\sqrt N\Delta'_N(y_0)| \le \log N$, we can find an
$\chi'>0$ such that for $x\in (-\infty, -N^\vep) \cup (N^\vep, \infty)$
$$
-\frac{\chi}2 x^2
+\sqrt N \Delta'_N(y_0)x+\frac 12 \Delta''_N(\xi)x^2 \le -\frac{\chi'}2 x^2
$$
whenever $N \ge \max\{N_0(\omega), N_1(\omega), N_2(\omega)\}$. Hence, since
$$
\int_{-\infty}^{-N^\vep}\exp\left(-\frac{\chi'}2\frac{x^2}2\right)dx+
\int^{\infty}_{N^\vep}\exp\left(-\frac{\chi'}2\frac{x^2}2\right)dx \to 0
$$
as $N \to \infty$ we obtain
\be \label{eq:ANiso}
\int_{A_N} \exp\left(N G_N(y)\right) g(y)dy=o\left(\frac 1{\sqrt N} e^{N G_N(y_0)}\right).
\ee
Let us now turn to the remaining integral, which is the main contribution. Using Taylor's formula and recalling that $G'(y) = 0$, we have
$$
G(y) = G(y_0) + \frac {G''(\kappa)} 2 (y-y_0)^2,
$$
where $\kappa = \kappa(y)$ is a number between $y$ and $y_0$. Recall also the Taylor expansion~\eqref{eq:taylor_Delta_N} of $\Delta_N(y)$ where $\xi=\xi(y)$ is also between $y$ and $y_0$.
Using these expansions and the substitution $y= y_0+\frac x{\sqrt N}$  we obtain
\begin{align}\label{eq:newint}
&\int_{y_0-\frac{N^\vep}{\sqrt N}}^{y_0+\frac{N^\vep}{\sqrt N}} \exp\left(N G_N(y)\right) g(y)dy\notag\\
&=
\int_{y_0-\frac{N^\vep}{\sqrt N}}^{y_0+\frac{N^\vep}{\sqrt N}} \exp\left(N G(y)
+N\left(\Delta_N(y_0)+\Delta'_N(y_0)(y-y_0)+\frac 12 \Delta''_N(\xi)(y-y_0)^2\right)\right)g(y)dy\notag\\
&=e^{N G_N(y_0)} \int_{y_0-\frac{N^\vep}{\sqrt N}}^{y_0+\frac{N^\vep}{\sqrt N}}
\exp\left(N\frac{G''(\kappa)}2 (y-y_0)^2
+N\Delta'_N(y_0)(y-y_0)+\frac N2 \Delta''_N(\xi)(y-y_0)^2\right)g(y)dy \notag\\
&=\frac 1{\sqrt N}\exp\left(NG_N(y_0)\right)\int_{-N^\vep}^{N^\vep} \exp\left(\frac{G''(\kappa)}2 x^2
+\sqrt N \Delta'_N(y_0)x+\frac 12 \Delta''_N(\xi)x^2\right)g\left(y_0+\frac x{\sqrt N}\right) dx \notag\\
&=\frac {g(y_0)}{\sqrt N}\exp\left(NG_N(y_0)\right)\int_{-N^\vep}^{N^\vep} \exp\left(\frac{G''(\kappa)}2 x^2
+\sqrt N \Delta'_N(y_0)x\right)dx \cdot (1+o(1)),
\end{align}
since, by~\eqref{eq:sup_Delta_N''},  $x^2 |\Delta''_N(\xi)| \leq N^{2\vep + \vep'}/\sqrt N \to 0$ uniformly for $\xi \in [y_0-r,y_0+r]$, provided $\vep$ and $\vep'$ are small enough.  We also used that $g$ is continuous at $y_0$ with $g(y_0)\neq 0$.

Let $\eta:= - G''(y_0)>0$. Since $\kappa= \kappa (y)$ belongs to  the interval  $[y_0-\frac{N^\vep}{\sqrt N}, y_0+\frac{N^\vep}{\sqrt N}]$ and $G$ has continuous third derivative,  we have $x^2 |G''(\kappa) + \eta| < C\cdot \frac{N^{3\vep}}{\sqrt N}\to 0$ uniformly over $x\in [-N^\vep, N^\vep]$, provided $\vep$ is small enough. It follows that
\begin{multline*}
\int_{y_0-\frac{N^\vep}{\sqrt N}}^{y_0+\frac{N^\vep}{\sqrt N}} \exp\left(N G_N(y)\right) g(y)dy
\\=
\frac {g(y_0)}{\sqrt N}\exp\left(NG_N(y_0)\right) \int_{-N^\vep}^{N^\vep} \exp\left(-\frac \eta 2 x^2
+\sqrt N \Delta'_N(y_0)x\right)dx \cdot (1+o(1)),
\end{multline*}
The remaining integral is then computed by completing the square in the exponent:
\begin{align}\label{eq:remainint}
& \int_{-N^\vep}^{N^\vep} \exp\left(-\frac{\eta}2 x^2+\sqrt N \Delta'_N(y_0)x\right)dx \notag\\
=&
\exp\left(\frac{N}{2 \eta}(\Delta'_N(y_0))^2\right)
\int_{-N^\vep}^{N^\vep} \exp\left(-\frac{\eta}2 \left(x-\frac{\sqrt N}\eta\Delta'_N(y_0)\right)^2\right)dx \notag\\
=& \exp\left(\frac{N}{2 \eta}(\Delta'_N(y_0))^2\right)\int_{-N^\vep-\frac{\sqrt N}\eta \Delta'_N(y_0)}^{N^\vep-\frac{\sqrt N}\eta \Delta'_N(y_0)}
\exp\left(-\frac \eta 2 z^2\right)dz\notag\\
=&\frac{\exp\left(\frac{N}{2 \eta}(\Delta'_N(y_0))^2\right)}{\sqrt \eta}
\int_{-\frac{N^\vep}{\sqrt \eta}-\frac{\sqrt N}{\eta^{3/2}} \Delta'_N(y_0)}^{\frac{N^\vep}{\sqrt \eta}
-\frac{\sqrt N}{\eta^{3/2}} \Delta'_N(y_0)}\exp\left(-\frac{\zeta^2}2\right) d\zeta.
\end{align}
Recall that we already showed that $\sqrt N |\Delta'_N(y_0)| \le \log N$ for almost all realizations of the $(h_i)$ and an
$N \ge N_1(\omega)$ where again $\omega$ is the randomness for the $(h_i)$. This implies that
$$
\pm \frac{N^\vep}{\sqrt \eta}-\frac{\sqrt N}{\eta^{3/2}} \Delta'_N(y_0)\to \pm \infty
$$
almost surely. Therefore,
$$
\int_{-\frac{N^\vep}{\sqrt \eta}-\frac{\sqrt N}{\eta^{3/2}} \Delta'_N(y_0)}^{\frac{N^\vep}{\sqrt \eta}
-\frac{\sqrt N}{\eta^{3/2}} \Delta'_N(y_0)}\exp\left(-\frac{\zeta^2}2\right) d\zeta \to \sqrt{2 \pi}
$$
almost surely.

Altogether we have seen that
$$
\int_{y_0-\frac{N^\vep}{\sqrt N}}^{y_0+\frac{N^\vep}{\sqrt N}} \exp\left(N G_N(y)\right) g(y)dy
\sim
g(y_0)\sqrt{ \frac{2 \pi}{N\eta}}
\exp\left(N \left( G_N(y_0)+\frac 1{2 \eta} (\Delta'_N(y_0))^2\right)\right).
$$
Recalling~\eqref{eq:intrhs} (in which we can make $-a$ and $b$ large enough such that the $O$-term on the right-hand side is negligible with respect to $\eee^{N G_N(y_0)}$) and~\eqref{eq:ANiso}, we conclude that
\be
\int_{-\infty}^\infty \exp\left(N G_N(y)\right) g(y)dy\sim g(y_0)\sqrt{ \frac{2 \pi}{N\eta}}
\exp\left(N \left( G_N(y_0)+\frac 1{2 \eta} (\Delta'_N(y_0))^2\right)\right).
\ee
The same argument applies upon replacing $g(y)$ by $1$ and gives
\be
\int_{-\infty}^\infty \exp\left(N G_N(y)\right)dy\sim \sqrt{ \frac{2 \pi}{N\eta}}
\exp\left(N \left( G_N(y_0)+\frac 1{2 \eta} (\Delta'_N(y_0))^2\right)\right).
\ee
With these asymptotics we get back to \eqref{eq:rep_marg} to see that
\begin{align}\label{eq:finresult1}
\mu_{N,k}(\tilde \sigma)&=
\frac{ \int_{-\infty}^\infty \exp\left(N G_N(y)\right) g(y) dy}{\int_{-\infty}^\infty \exp\left(N G_N(y)\right) dy} \notag \\
=&  g(y_0)(1+o(1))\notag\\
=&  \prod_{i=1}^k\frac{\exp\left( \sigma_i\left(\sqrt \beta y_0+\beta h_i \right)\right)}
{2\cosh\left(\sqrt \beta y_0+\beta h_i \right)}(1+o(1))
\end{align}
with an $o(1)$-term that is necessarily uniform in the $(\sigma_1, \ldots,\sigma_k)\in\{-1,+1\}^k$.
Notice that for any fixed $(\sigma_1, \ldots,\sigma_k)$ the right-hand side in \eqref{eq:finresult1} agrees with
$\varrho^{(k)}(\sigma_1, \ldots,\sigma_k)$ as defined in Theorem \ref{theo:main1}.
Therefore
\begin{align}\label{eq:entropy}
H(\mu_{N,k}|\varrho^{(k)})=&\sum_{\tilde \sigma \in \{-1,+1\}^k} \log\left(\frac{\mu_{N,k}(\tilde \sigma)}{\varrho^{(k)}(\tilde \sigma)}\right)\mu_{N,k}(\tilde \sigma)\notag\\
=& \sum_{\tilde \sigma \in \{-1,+1\}^k} \log\left(\frac{\varrho^{(k)}(\tilde \sigma)(1+o(1)}{\varrho^{(k)}(\tilde \sigma)}\right)\mu_{N,k}(\tilde \sigma)\\
=& o(1)\sum_{\tilde \sigma \in \{-1,+1\}^k}\mu_{N,k}(\tilde \sigma)\notag
\end{align}
and the right-hand side converges to $0$, as $N\to \infty$.
\end{proof}

\begin{remark}\label{rem:increasing}
If $k=o(N^{\frac 12 - \vep})$ (for any fixed $\vep >0$) is allowed to increase with $N$, then the above proof applies with minor changes. Besides the condition $0\leq g(y)\leq 1$ which is clearly fulfilled we need to justify~\eqref{eq:newint}. To this end, we need to show that
$$
\lim_{N\to\infty} \frac{g\left(y_0+\frac x{\sqrt N}\right)}{g(y_0)} = 1
$$
uniformly in $|x|\leq N^\vep$ and $(\sigma_1,\ldots, \sigma_k)\in \{-1,1\}^k$. Taking the logarithms and putting $y = y_0+\frac x{\sqrt N}$, this amounts to showing that
$$
\sum_{i=1}^k \log(1 + e^{-2\sigma_i(\sqrt \beta y_0 + \beta h_i)}) - \log(1 + e^{-2\sigma_i(\sqrt \beta y + \beta h_i)}) \to 0
$$
uniformly in $y$ such that $|y-y_0|\leq  N^{\vep - \frac 12}$ and $(\sigma_1,\ldots, \sigma_k)\in \{-1,1\}^k$. Since the function $\log (1 + e^{-x})$ is $1$-Lipschitz, we can estimate the sum on the left-hand side  by $2 k \sqrt \beta |y-y_0|$, which converges to $0$ due to our assumption on the growth of $k$.
\end{remark}

Let us next see what changes when the maximum in the above proof is degenerate.

\begin{proof}[Proof of Theorem \ref{theo:main1b}]
Recall that we are still in the situation where the function $G$ possesses a unique maximum $y_0$, but this time $y_0$ is
$n$-degenerate for some $n \ge 2$. We follow the proof of Theorem \ref{theo:main1} until \eqref{eq:intrhs}. However, now we
decompose
$$[y_0-r,y_0+r]= A_N \mathop{\dot{\cup}} \left(y_0-N^{\vep-1/(2n)}, y_0+N^{\vep-1/(2n)}\right),$$
where, of course,
$$
A_N:=[y_0-r,y_0+r]\setminus \left(y_0-N^{\vep-\frac 1{2n}}, y_0+N^{\vep-\frac 1{2n}}\right).
$$
Let us again analyze the contribution of $A_N$ to the integral:
Keeping our definition of $\Delta_N$ from \eqref{eq:Delta} we can simply copy \eqref{eq:intGGN}. However,
now we take a Taylor expansion of $\Delta_N$ up to order $2n$ to obtain
\begin{multline*}
\Delta_N(y)=\Delta_N(y_0)+\Delta'_N(y_0)(y-y_0)+\frac 12 \Delta''_N(y_0)(y-y_0)^2+\ldots \\+\frac 1{(2n-1)!}\Delta^{(2n-1)}_N(y_0)
(y-y_0)^{2n-1}+\frac 1{(2n)!}\Delta^{(2n)}_N(\xi)
(y-y_0)^{2n}
\end{multline*}
for some $\xi \in [y_0,y]$ or $\xi \in[y,y_0]$.

Note that with the notation of the proof of Theorem \ref{theo:main1} the terms $\Delta^{(l)}_N(y_0), l=1, \ldots 2n-1$ are of the form
$$
\Delta^{(l)}_N(y_0)=\frac 1 N\sum_{i=1}^N X_i^{(l)}(y_0)
$$
where for fixed $l$ the random variables $(X_i(y_0))_{i=1}^N$ are bounded i.i.d. random variables. Hence by the Law of the Iterated
Logarithm we obtain that
$$
\lvert\Delta^{(l)}_N(y_0)\rvert \le \frac{\log N}{\sqrt N} \quad \mbox{for all } l=1, \ldots 2n-1
$$
and for all $N \ge N_0(\omega)$ $\P$-almost surely.

Also, by the same arguments that we applied to obtain $\lvert\Delta^{''}_N(\xi)\rvert \le \frac{N^{\vep'}}{N}$  uniformly
in $\xi \in [y_0-r,y_0+r]$ in the proof of Theorem
\ref{theo:main1} we can see that here
$$
\lvert\Delta^{(2n)}_N(\xi)\rvert \le \frac{N^{\vep'}}{\sqrt N} \quad \mbox{uniformly in } \xi \in [y_0-r,y_0+r]
$$
for some $\vep'>0$ sufficiently small and and for all $N \ge N_1(\omega)$ $\P$-almost surely.

Recall that the derivatives of order up to $2n-1$ of $G$ in $y_0$ vanish.
By a computation similar to \eqref{eq:longcomp} we obtain that there is $\eta>0$ such that
\begin{align}\label{eq:longcomp2}
&\int_{A_N} \exp\left(N G_N(y)\right) g(y)dy \\
\leq & \int_{A_N} \exp\left(N \left(G(y_0)- \eta (y-y_0)^{2n}\right)\right. \notag\\
& \qquad\qquad\left.
+N\left(G_N(y_0)-G(y_0)\right)+N\Delta'_N(y_0)(y-y_0)+\ldots+\frac N{(2n)!} \Delta^{(2n)}_N(\xi)(y-y_0)^{2n}\right)g(y)dy\notag\\
=&e^{N G_N(y_0)} \int_{A_N} \exp\left(-N \eta (y-y_0)^{2n}
+N\Delta'_N(y_0)(y-y_0)+\ldots+\frac N{(2n)!} \Delta^{(2n)}_N(\xi)(y-y_0)^{2n}\right)g(y)dy\notag\\
\le &N^{-\frac 1{2n}} e^{N G_N(y_0)}\notag\\
&\quad\left(\int_{-\infty}^{-N^\vep} \exp\left(-\eta x^{2n}
+N^{1-\frac 1{2n}}\Delta'_N(y_0)x+N^{1-\frac 2{2n}} \frac 1 {2!}\Delta''_N(y_0)x^2+\ldots+\frac 1{(2n)!} \Delta^{(2n)}_N(\xi)x^{2n}\right)dx\right.
\notag\\
&\qquad +
\left.\int^{\infty}_{N^\vep} \exp\left(-\eta x^{2n}
+N^{1-\frac 1{2n}}\Delta'_N(y_0)x+N^{1-\frac 2{2n}}\frac 1 {2!} \Delta''_N(y_0)x^2+\ldots+\frac 1{(2n)!} \Delta^{(2n)}_N(\xi)x^{2n}\right)dx\right)
\notag
\end{align}
by setting $x=N^{\frac1{2n}}(y-y_0)$. We used that $0\leq g(y)\leq 1$.

Take some (small) $c>0$. By the above considerations, for each $l=1, \ldots, 2n-1$ we have that for all $N$ large enough and almost all realizations
of the $(h_i)$
$$
\lvert N^{1-\frac l{2n}} \Delta^{(l)}_N(y_0)x^{l} \rvert\le \lvert N^{\frac{n-l}{2n}}\log Nx^{l} \rvert \le c x^{2n}
$$
for $|x| \ge N^{\vep}$, if we take $\vep > \frac{n-1}{2n(2n-1)}$ (later we shall also need that $\vep < \frac 1{2n}$; note that these conditions are consistent).
Also note that $ \Delta^{(2n)}_N(\xi) \to 0$ uniformly in $\xi \in [y_0-r,y_0+r]$. Therefore there is $\eta'>0$ such that
\begin{multline*}
\int_{A_N} \exp\left(N G_N(y)\right) g(y)dy \\
\le N^{-\frac 1{2n}} e^{N G_N(y_0)}
\left(\int_{-\infty}^{-N^\vep} \exp\left(- \eta' x^{2n}\right)dx+\int^{\infty}_{N^\vep} \exp\left(-\eta' x^{2n}\right)dx\right)
\end{multline*}
for all $N$ large enough and $\P$-almost surely in the $(h_i)$. Hence,
\begin{equation}\label{eq:tech_123}
\int_{A_N} \exp\left(N G_N(y)\right) g(y)dy=o( N^{-\frac 1{2n}} e^{N G_N(y_0)}e^{-\eta' N^{2n\vep}} ).
\end{equation}

For the remaining integral over $[y_0-N^{\vep-\frac 1{2n}}, y_0+N^{\vep-\frac 1{2n}}]$  (where we shall assume that $\vep < \frac 1 {2n}$) we perform the same
transformations as above (similar to what we did in \eqref{eq:newint}) to obtain, for some $\kappa = \kappa(y)$ between $y_0$ and $y$,
\begin{align}\label{eq:innerinttheo3}
&\int_{y_0-N^{\vep-\frac 1{2n}}}^{y_0+N^{\vep-\frac 1{2n}}} \exp\left(N G_N(y)\right) g(y)dy\notag\\
=&
\int_{y_0-N^{\vep-\frac 1{2n}}}^{y_0+N^{\vep-\frac 1{2n}}}  \exp\left(N \left( G(y_0) + \frac{G^{(2n)}(\kappa)}{(2n)!} (y-y_0)^{2n} \right)\right. \notag\\
& \qquad\qquad\left.
+N\left(G_N(y_0)-G(y_0)\right)+N\Delta'_N(y_0)(y-y_0)+\ldots+\frac N{(2n)!} \Delta^{(2n)}_N(\xi)(y-y_0)^{2n}\right)g(y)dy\notag\\
=&N^{-\frac 1{2n}}\exp\left(NG_N(y_0)\right)\\
&\quad\times
\int_{-N^\vep}^{N^\vep} e^{\frac{G^{(2n)}(\kappa)}{(2n)!} x^{2n}
+N^{1-\frac 1{2n}}\Delta'_N(y_0)x+N^{1-\frac 2{2n}}\frac 1 {2!}\Delta''_N(y_0)x^2+\ldots+\frac 1{(2n)!} \Delta^{(2n)}_N(\xi)x^{2n}}
g(y_0+xN^{-\frac 1 {2n}})dx\notag\\
=& N^{-\frac 1{2n}}\exp\left(NG_N(y_0)\right) g(y_0)\notag \\
&\quad \times \int_{-N^\vep}^{N^\vep} e^{\frac{G^{(2n)}(\kappa)}{(2n)!} x^{2n}
+N^{1-\frac 1{2n}}\Delta'_N(y_0)x+N^{1-\frac 2{2n}}\frac 1 {2!}\Delta''_N(y_0)x^2+\ldots+\frac 1{(2n-1)!} \Delta^{(2n-1)}_N(\xi)x^{2n-1}}
dx \cdot (1+o(1)).\notag
\end{align}
Unfortunately the integral on the right-hand side
cannot be evaluated as easily as in \eqref{eq:remainint}, because the argument of the exponent is a polynomial of order at least 4, and
hence there is no square to complete in the exponent. However, we do not need the value of this integral for the proof of Theorem
\ref{theo:main1b} since it will occur in both, the numerator and the denominator, in \eqref{eq:rep_marg} (the reason why we computed
the corresponding integral in \eqref{eq:remainint} is that we will need this step in the proof of Theorem \ref{theo:main2}).
More precisely, for some (small) $c>0$ we have the lower estimate:
\begin{align*}
&\int_{-N^\vep}^{N^\vep} e^{\frac{G^{(2n)}(\kappa)}{(2n)!} x^{2n}
+N^{1-\frac 1{2n}}\Delta'_N(y_0)x+N^{1-\frac 2{2n}}\frac 1 {2!}\Delta''_N(y_0)x^2+\ldots+\frac 1{(2n-1)!} \Delta^{(2n-1)}_N(\xi)x^{2n-1}}
dx\\
&\geq
c \int_{1}^{2}e^{-c N^{1- \frac 1 {2n}}(\log N) /\sqrt N} dx
\geq c e^{-c N^{(2n-1)\eps}}
\end{align*}
since $\vep > \frac{n-1}{2n(2n-1)}$. Recall also that $(2n-1)\eps < 1$. This implies that the right-hand side of~\eqref{eq:tech_123} and the $O$-term in~\eqref{eq:intrhs} (where we take $r>0$ large enough) are negligible with respect to~\eqref{eq:innerinttheo3}. It follows that
\begin{multline*}
\int_{-\infty}^{+\infty} \exp\left(N G_N(y)\right) g(y)dy
=
N^{-\frac 1{2n}}\exp\left(NG_N(y_0)\right) g(y_0)\\
\times \int_{-N^\vep}^{N^\vep} e^{\frac{G^{(2n)}(\kappa)}{(2n)!} x^{2n}
+N^{1-\frac 1{2n}}\Delta'_N(y_0)x+N^{1-\frac 2{2n}}\frac 1 {2!}\Delta''_N(y_0)x^2+\ldots+\frac 1{(2n-1)!} \Delta^{(2n-1)}_N(\xi)x^{2n-1}}
dx \cdot (1+o(1)).
\end{multline*}
The same argument, with $g(y)$ replaced by $1$, yields
\begin{multline*}
\int_{-\infty}^{+\infty} \exp\left(N G_N(y)\right)dy
=
N^{-\frac 1{2n}}\exp\left(NG_N(y_0)\right)\\
\times \int_{-N^\vep}^{N^\vep} e^{\frac{G^{(2n)}(\kappa)}{(2n)!} x^{2n}
+N^{1-\frac 1{2n}}\Delta'_N(y_0)x+N^{1-\frac 2{2n}}\frac 1 {2!}\Delta''_N(y_0)x^2+\ldots+\frac 1{(2n-1)!} \Delta^{(2n-1)}_N(\xi)x^{2n-1}}
dx \cdot (1+o(1)).
\end{multline*}
Note that $\kappa$ (which is a function of $y$ and hence $x$ and $N$) is the same in both formulas. Using~\eqref{eq:rep_marg} we obtain:
\begin{align*}
\mu_{N,k}(\tilde \sigma) =&\sum_{\sigma_{k+1}, \ldots \sigma_N\in \{-1,+1\}}\frac{\exp\left(-\beta H_N(\sigma)\right)}{Z_N}\\
=&
\frac{ \int_{-\infty}^\infty \exp\left(N G_N(y)\right) g(y) dy}{\int_{-\infty}^\infty \exp\left(N G_N(y)\right) dy}\\
=& \frac{g(y_0)\int_{-N^\vep}^{N^\vep} e^{\frac{G^{(2n)}(\kappa)}{(2n)!} x^{2n}
+N^{1-\frac 1{2n}}\Delta'_N(y_0)x+N^{1-\frac 2{2n}}\frac 1 {2!}\Delta''_N(y_0)x^2+\ldots+\frac 1{(2n-1)!} \Delta^{(2n-1)}_N(\xi)x^{2n-1}}dx
(1+o(1))} {\int_{-N^\vep}^{N^\vep} e^{\frac{G^{(2n)}(\kappa)}{(2n)!} x^{2n}
+N^{1-\frac 1{2n}}\Delta'_N(y_0)x+N^{1-\frac 2{2n}}\frac 1 {2!}\Delta''_N(y_0)x^2+\ldots+\frac 1{(2n-1)!} \Delta^{(2n-1)}_N(\xi)x^{2n-1}}dx
(1+o(1))}\\
=&  g(y_0) (1+o(1))\\
=& (1+o(1)) \prod_{i=1}^k\frac{\exp\left( \sigma_i\left(\sqrt \beta y_0+\beta h_i \right)\right)}
{2\cosh\left(\sqrt \beta y_0+\beta h_i \right)}.
\end{align*}
The proof again concludes with the observation that for fixed $(\sigma_1, \ldots,\sigma_k)$ the right-hand side agrees with
$\varrho^{(k)}(\sigma_1, \ldots,\sigma_k)$ as defined in Theorem \ref{theo:main1b} together with the same observation as in \eqref{eq:entropy}.
\end{proof}
Let us finally turn to the proof of Theorem \ref{theo:main2}.
\begin{proof}
Recall that now $G$ possesses several non-degenerate global
maxima $y_0, \ldots ,y_s$. Again we follow the proof of Theorem \ref{theo:main1} by rewriting $\mu_{N,k}$ as in
\eqref{eq:rep_marg} until \eqref{eq:intrhs}. Here we take the interval $[a,b]$ so large that $(a,b)$ contains all the points $y_0, \ldots,y_s$.

First, for each $j=0, \ldots ,s$ take an interval $[a_j,b_j]\subset [a,b]$ such that $y_j$ is the only maximizer of $G$ in the interval $[a_j,b_j]$, for each $j=0, \ldots, s$.
Then for any bounded function $g$ with $0\leq g(y)\leq 1$ and $A:= [a,b]\setminus \bigcup_{j=0}^s [a_j,b_j]$
the contribution of
$$\int_A \exp\left(N G_N(y)\right) g(y) dy \qquad \mbox{to }\quad\int_{a}^b \exp\left(N G_N(y)\right) g(y) dy$$
is negligible as $N \to \infty$. Indeed, since $G$ is continuous on the compact interval $[a,b]$ and $\bigcup_{j=0}^s [a_j,b_j]$
contains all the global maxima of $G$
$$
\sup_{y \in [a,b]} G(y)-\sup_{y \in A} G(y)>0
$$
and obviously
$$
\int_A \exp\left(N G_N(y)\right) g(y) dy \le (b-a) e^{N \sup_{y \in A} G(y)}
= O\left(\frac{e^{NG(y_0)}}{\sqrt N}\right).
$$
The right-hand side will be seen to be smaller or equal than the order of $\int_{a}^b \exp\left(N G_N(y)\right) g(y) dy$.

Next, for each $j=0, \ldots, s$ and
$\vep>0$ small enough we take an open neighborhood $(y_j-\frac{N^{\vep}}{\sqrt N},y_j+\frac{N^{\vep}}{\sqrt N})$ of $y_j$. For $N$ large enough all this neighborhood will be
contained in $[a_j,b_j]$.
We want to see further  that the integral $\int_{a}^b \exp\left(N G_N(y)\right) g(y) dy$
gets its entire contribution (up to a negligible part) from $\bigcup_{j=0}^s (y_j-\frac{N^{\vep}}{\sqrt N},y_j+\frac{N^{\vep}}{\sqrt N})$.
However, this can be done by repeating the arguments that were used in the proof of Theorem \ref{theo:main1} $s+1$ times
and by replacing
$A_N$ by $[a_j,b_j]\setminus (y_j-\frac{N^{\vep}}{\sqrt N},y_j+\frac{N^{\vep}}{\sqrt N})$ for each $j=0, \ldots, s$ in the
steps from \eqref{eq:AN} to \eqref{eq:ANiso}.

It remains to compute the asymptotics of the integrals
$$
\int_{y_j-\frac{N^{\vep}}{\sqrt N}}^{y_j-\frac{N^{\vep}}{\sqrt N}} \exp\left(N G_N(y)\right) g(y) dy \qquad j=0, \ldots, s.
$$
Applying the computations in \eqref{eq:newint} and \eqref{eq:remainint} to each of these integrals we find that
\be
\int_{y_j-\frac{N^{\vep}}{\sqrt N}}^{y_j-\frac{N^{\vep}}{\sqrt N}} \exp\left(N G_N(y)\right) g(y)dy = g(y_j)\sqrt{ \frac{2 \pi}{N\eta_j}}
\exp\left(N \left( G_N(y_j)+\frac {(\Delta'_N(y_j))^2}{2 \eta_j} \right)\right)(1+o(1))
\ee
where $\eta_j = - G''(y_j) >0$ by assumption on $G$ and $y_j$. Therefore, making use of \eqref{eq:rep_marg} we obtain for any
fixed $k\in \N$, any
fixed $\tilde \sigma = (\sigma_1, \ldots, \sigma_k)$ and $\P_h$-almost any realization of the $(h_i)$ that
\begin{align}\label{eq:severalmax}
\mu_{N,k}(\tilde \sigma)=& \frac{\sum_{j=0}^s  g(y_j)\sqrt{ \frac{2 \pi}{N\eta_j}}
\exp\left(N \left( G_N(y_j)+\frac 1{2 \eta_j} (\Delta'_N(y_j))^2\right)\right)(1+o(1))}{\sum_{j=0}^s  \sqrt{ \frac{2 \pi}{N\eta_j}}
\exp\left(N \left( G_N(y_j)+\frac 1{2 \eta_j} (\Delta'_N(y_j))^2\right)\right)(1+o(1))}\notag \\
=& \frac{\sum_{j=0}^s  g(y_j)\sqrt{ \frac{2 \pi}{N\eta_j}}
\exp\left(N \left( G(y_j)+\Delta_N(y_j)+\frac 1{2 \eta_j} (\Delta'_N(y_j))^2\right)\right)(1+o(1))}{\sum_{j=0}^s  \sqrt{ \frac{2 \pi}{N\eta_j}}
\exp\left(N \left( G(y_j)+\Delta_N(y_j)+\frac 1{2 \eta_j} (\Delta'_N(y_j))^2\right)\right)(1+o(1))}\notag \\
=&\frac{\sum_{j=0}^s  g(y_j)\sqrt{ \frac{1}{\eta_j}}
\exp\left(N\Delta_N(y_j)+\frac N{2 \eta_j} (\Delta'_N(y_j))^2\right)(1+o(1))}{\sum_{j=0}^s  \sqrt{ \frac{1}{\eta_j}}
\exp\left(N\Delta_N(y_j)+\frac N{2 \eta_j} (\Delta'_N(y_j))^2\right)(1+o(1))}
\end{align}
since the values of $G(y_j), j=0, \ldots, s$ are the same. 

Now observe that by the Law of the Iterated Logarithm, similar to \eqref{eq:LIL1}
$$
\frac N{2 \eta_j} (\Delta'_N(y_j))^2\le \log^2(N)
$$
for all $j=0, \ldots, s$ and $N$ large enough (depending on $j$ and the realization of the $(h_i)$).

On the other hand, by the Central Limit Theorem, the random vector with components
\be
\sqrt N \Delta_N(y_j):=
\frac 1 {\sqrt N} \sum_{i=1}^N \left(\log \cosh\left(\sqrt \beta y_j+\beta h_i \right)
-\E_h \left[\log \cosh\left(\sqrt \beta y_j+\beta h\right)\right]\right),
\end{equation}
where $j=0, \ldots, s$, converges in distribution to a multivariate Gaussian random vector
$(\xi_j)_{j=0}^s$ with $\E \xi_j = 0$. Let us argue that $\xi_{j_1}-\xi_{j_2}$ has nonzero variance provided $j_1\neq j_2$. For concreteness, we consider $\xi_0 - \xi_1$. Without loss of generality, let $y_0 < y_1$.   By the Central Limit Theorem, $\mathrm{Var}\, (\xi_0 - \xi_1) = \mathrm{Var}\, q(h_1)$, where
$$
q(h) := \log \cosh\left(\sqrt \beta y_0+\beta h \right) - \log \cosh\left(\sqrt \beta y_1+\beta h \right).
$$
The derivative of $q(h)$ in $h$ is
$$
q'(h) = \beta \tanh \left(\sqrt \beta y_0+\beta h \right) - \beta \tanh \left(\sqrt \beta y_1+\beta h \right) <0
$$
since $\tanh(x)$ is an increasing function. Note that $h_1$ is not a.s.\ constant because we assume that $\E h_1^2 \neq 0$ and $G$ has more than one global maximum, which is not possible in the classical Curie-Weiss model if the external field is non-zero. It follows that $q$ is not constant on the support of $h_1$ and hence $\Var q(h_1) >0$. Hence $\mathrm{argmax}_{j=0,\ldots, s} \xi_j$ is uniquely defined with probability $1$.

Let  ${J_N}:= \mathrm{argmax}_{j\in 0,\ldots, s} \sqrt N \Delta_N(y_j)$. If, for some value of $N$, there is no unique maximizer of the variables $\sqrt N \Delta_N(y_j)$, $j=0,\ldots,s$, we just take any index
$J_N$ in which $\sqrt N \Delta_N(y_j)$ is maximized.

Let $\delta >0$. Recall that $(\sqrt N \Delta_N(y_j))_{j=0,\ldots, s}$ converges weakly to $(\xi_j)_{j\in 0,\ldots, s}$ and that the latter random vector has a.s.\ pairwise different components. By the continuous mapping theorem, there is $\eps>0$ such that for all sufficiently large $N$, we have
$$
\mathbb P\left(A_N\right) >1-\frac \delta2,
\text{ where }
A_N:=  \left\{\sqrt N \Delta_N(y_{J_N}) - \max_{j\neq J_N} \sqrt N \Delta_N(y_{j}) > \eps\right\}.
$$
We have shown above that, for sufficiently large $N$,
$$
\mathbb P\left(B_N\right) >1-\frac \delta2,
\text{ where }
B_N:=
\left\{\max_{j=0,\ldots, s} \frac N{2 \eta_j} (\Delta'_N(y_j))^2\le \log^2(N)\right\}.
$$
It follows that for all $j\neq J_N$ on the event $A_N\cap B_N$  we have
\begin{align*}
\exp\left(N\Delta_N(y_j)+\frac N{2 \eta_j} (\Delta'_N(y_j))^2\right)
&\leq
\exp\left(N\Delta_N(y_{J_N}) - \sqrt N \eps + \frac N{2 \eta_j} (\Delta'_N(y_j))^2\right)\\
&=
o\left(\exp\left(N\Delta_N(y_{J_N})+\frac N{2 \eta_{J_N}} (\Delta'_N(y_{J_N}))^2\right)\right).
\end{align*}

%

Thus, on the event $A_N\cap B_N$ (whose probability bounded from below by $1-\delta$),  both, the numerator and the denominator
in \eqref{eq:severalmax} are dominated by the contribution of the (random) summand with index $j=J_N$. That is to say
\begin{align*}
\mu_{N,k}(\tilde \sigma)
=&\frac{g(y_{J_N})\sqrt{ \frac{1}{\eta_{J_N}}}
\exp\left(N\Delta_N(y_{J_N})+\frac N{2 \eta_{J_N}} (\Delta'_N(y_{J_N}))^2\right)(1+o(1))}{\sqrt{ \frac{1}{\eta_{J_N}}}
\exp\left(N\Delta_N(y_{J_N})+\frac N{2 \eta_{J_N}} (\Delta'_N(y_{J_N}))^2\right)(1+o(1))}\\
=&g(y_{J})(1+o(1))
\end{align*}
Again for any $\tilde\sigma \in \{-1,+1\}^k$, by definition of the function $g$ we have $g(y_{{J_N}})=\varrho_{J_N}^{(k)}(\tilde \sigma)$ and the proof finishes along the lines
of \eqref{eq:entropy} with the slight difference that the convergence in the theorem is now in probability.
\end{proof}

\section{Example: The Random-Field Curie-Weiss model with dichotomous external field}\label{sec:example}
In this section we will specify our results to an example in which the phase diagram has been studied in detail (see \cite{AGZ92},
Section 5), the Random-Field Curie-Weiss model with dichotomous external field. Here the random variables $(h_i)_i$ are i.i.d. with
$$
\P_h(h_i=\mathfrak{h})=\P_h(h_i=-\mathfrak{h})=\frac 12
$$
for a $\mathfrak{h} \ge 0$.
Even in this very simple situation the phase diagram is much richer than in the situation of the ordinary Curie-Weiss model, which we
re-obtain by setting $\mathfrak{h}=0$. Of course, in the Curie-Weiss model, we re-obtain the known propagation of chaos-results, even
though not exactly the same {\it increasing} propagation of chaos result as in \cite{BAZ_chaos}.

For $\mathfrak{h}>0$ the size of $\mathfrak{h}$ influences the phase diagram: If $\mathfrak{h}\ge\frac 12 $, then $G$ has a unique
maximum at $y=0$. On the other hand, if $\mathfrak{h}<\frac 12$, there is an increasing function $f:[0,\frac 12) \to \R$ such that
$f(0)=1$ and $f(x) \to \infty$, whenever $x \to \frac 12$ such that
\begin{enumerate}
\item If $\beta < f(\mathfrak{h})$, then $G$ has an unique non-degenerate global maximum at $y=0$.
\item If $\beta= f(\mathfrak{h})$ and  $\mathfrak{h} \le \frac 2 3 \mathrm{arcosh}\left( \sqrt{\frac32}\right)$,
then $G$ has an unique global maximum at $y=0$.
\item $\beta= f(\mathfrak{h})$ and  $\mathfrak{h} > \frac 2 3 \mathrm{arcosh}\left( \sqrt{\frac32}\right)$,
then the function $G$
has three non-degenerate global maxima, one of which is $y=0$ while the other two, say $y_1$ and $y_{-1}$ are symmetric around the origin,
hence $y_1=-y_{-1}$.
\item If $\beta> f(\mathfrak{h})$ then $G$ has two non-degenerate symmetric global minima $y_1$ and $y_{-1}=-y_1$.
\end{enumerate}
Consequently, the system is in a paramagnetic phase, whenever $\mathfrak{h} \ge \frac 12$ or both
$\mathfrak{h}<\frac 12$ and $\beta < f(\mathfrak{h})$.

In case 4) above the system is in a ferromagnetic phase.

In the region, where $\mathfrak{h}<\frac 12$ and $\beta= f(\mathfrak{h})$ there are again three scenarios. If
$\mathfrak{h} <\frac 2 3 \mathrm{arcosh}\left( \sqrt{\frac32}\right)$ there is a second order phase transition, i.e.\
the unique global maximum of $G$ at $y=0$ is $2$-degenerate. If
$\mathfrak{h} > \frac 2 3 \mathrm{arcosh}\left( \sqrt{\frac32}\right)$ there is a first order phase transition, i.e.\
the three global maxima of the function $G$ are non-degenerate. Finally, the point $\mathfrak{h} =\frac 2 3 \mathrm{arcosh}\left( \sqrt{\frac32}\right)$
and $\beta= f(\mathfrak{h})$ is a so-called tri-critical point, i.e. the unique global maximum of $G$ at $y=0$ is $3$-degenerate.

Therefore, our results in Theorems \ref{theo:main1}, \ref{theo:main1b}, and \ref{theo:main2} can be translated to this example
in the following way:

\begin{corollary}
For the Random-Field Curie-Weiss model with dichotomous external field the following propagation of chaos results hold true
\begin{enumerate}
\item In the situation where
\begin{itemize}
\item Either $\mathfrak{h} \ge \frac 12$,
\item or $\mathfrak{h}<\frac 12$ and $\beta < f(\mathfrak{h})$,
\item or
$\beta= f(\mathfrak{h})$ and  $\mathfrak{h} \le \frac 2 3 \mathrm{arcosh}\left( \sqrt{\frac32}\right)$
\end{itemize}
let
\be \label{eq:rho0}
\varrho_0^{(k)}(\sigma_1, \ldots, \sigma_k):= \prod_{i=1}^k\frac{\exp\left(\sigma_i\left(\beta h_i \right)\right)}
{2\cosh\left(\beta h_i \right)}.
\ee
Then, for any fixed $k$ and almost all realizations of the random external field $(h_i)$ we have
$$
H(\mu_{N,k}|\varrho_0^{(k)}) \to 0 \qquad \mbox{as } N\to \infty.
$$
\item If $\mathfrak{h}<\frac 12$ and $\beta> f(\mathfrak{h})$, define
\be \label{eq:rhopm1}
\varrho_j^{(k)}(\sigma_1, \ldots, \sigma_k):= \prod_{i=1}^k\frac{\exp\left(\sigma_i\left(\sqrt \beta y_j
\beta h_i \right)\right)}
{2\cosh\left(\sqrt \beta y_j \beta h_i \right)}
\ee
for $j\in \{-1,+1\}$ and $y_1$ and $y_{-1}$ defined as in item 4) above.

Then, for almost every realization of the $(h_i)_{i \in \N}$ and every $N$
there is a random
index $J= J_N^{(k)}(h_1,\ldots, h_N)\in \{0, \ldots, s\}$ such that
for all $\sigma_1, \ldots, \sigma_k \in \{-1,+1\}$
\begin{equation}\label{eq:claim1}
\left|\mu_{N,k}(\sigma_1, \ldots, \sigma_k)-\varrho^{(k)}_{J_N^{(k)}}(\sigma_1, \ldots, \sigma_k)\right| \to 0
\end{equation}
as well as
\begin{equation}\label{eq:claim2}
H\left(\mu_{N,k} \vert \varrho^{(k)}_{J_N^{(k)}}\right) \to 0 \qquad \text{ as } N \to \infty,
\end{equation}
where the
convergence is in the sense of convergence in probability.

\item If $\beta= f(\mathfrak{h})$ and  $\mathfrak{h} > \frac 2 3 \mathrm{arcosh}\left( \sqrt{\frac32}\right)$ define
$\varrho_0^{(k)}$ and $\varrho_{\pm 1}^{(k)}$ as in \eqref{eq:rho0} and \eqref{eq:rhopm1}, respectively. Then, again,
for almost every realization of the $(h_i)_{i \in \N}$ and every $N$
there is a random
index $J = J_N^{(k)}(h_1,\ldots, h_N)\in \{0, \ldots, s\}$ such that
\eqref{eq:claim1} and~\eqref{eq:claim2} hold, 
where again, the
convergence is in the sense of convergence in probability.
\end{enumerate}
\end{corollary}

\end{document}